\documentclass[12pt,a4paper]{amsart}

\usepackage{docmute}
\usepackage{a4wide}
\usepackage[utf8]{inputenc}
\usepackage{amsmath,amssymb}

\usepackage[
backend=biber,
style=numeric,
sorting=nyt
]{biblatex}
\usepackage{xcolor}
\usepackage{booktabs}
\usepackage{hyperref}
\hypersetup{
	colorlinks,
	citecolor=black,
	filecolor=black,
	linkcolor=blue,
	urlcolor=black
}
\usepackage{stmaryrd}
\usepackage{url}
\usepackage{longtable}
\usepackage[figuresright]{rotating}
\usepackage{amsthm}
\usepackage{bbold}
\usepackage{enumerate}
\usepackage{geometry}
\newtheorem{definition}{Definition}
\newtheorem{theorem}[definition]{Theorem}
\newtheorem{proposition}[definition]{Proposition}

\newtheorem{lemma}[definition]{Lemma}
\newtheorem{fact}[definition]{Fact}

\newtheorem{corollary}[definition]{Corollary}
\newtheorem{question}{Question}

\newtheorem*{claim*}{Claim}

\newcommand{\0}{\emptyset}
\newcommand{\mc}{\mathcal}
\newcommand{\mbb}{\mathbb}

\newcommand{\RR}{\mathbb{R}}

\newcommand{\symdif}{\triangle}

\newcommand{\IFF}{\Leftrightarrow}
\newcommand{\IMP}{\Rightarrow}

\newcommand{\Ee}{\mc{E}}
\newcommand{\Ii}{\mc{I}}
\newcommand{\Jj}{\mc{J}}

\newcommand{\Mm}{\mc{M}}
\newcommand{\Nn}{\mc{N}}

\newcommand{\Aa}{\mc{A}}
\newcommand{\Bb}{\mc{B}}

\newcommand{\Gg}{\mc{G}}

\newcommand{\bez}{\backslash}
\newcommand{\se}{\subseteq}

\newcommand{\es}{\supseteq}

\newcommand{\rest}{\restriction}

\newcommand{\baire}{\omega^\omega}

\newcommand{\concat}{^{\frown}}

\newcommand{\tn}[1]{\textnormal{#1}}

\newcommand{\add}{\tn{add}}
\newcommand{\cof}{\tn{cof}}

\newcommand{\Bor}{\tn{Bor}}

\newcommand{\dom}{\textnormal{dom}}

\newcommand{\Perf}{\textnormal{Perf}}

\newcommand{\ZFC}{\textnormal{ZFC}}

\def\Bb{\mathcal{B}}

\def\bbr{\mathbb{R}}

\def\c{\mathfrak{c}}

\def\w{\omega}

\def\baire{\w^\w}
\def\perf{\rm Perf}
\def\bor{\rm Bor}
\def\ctbl{\rm ctbl}

\def\cof{\rm cof}

\def\then{\longrightarrow}
\def\iff{\longleftrightarrow}

\def\MA{{\bf MA }}

\def\ZFC{{\bf ZFC }}

\title{Around the Eggleston Theorem}

\author{Marcin Michalski}
\email{marcin.k.michalski@pwr.edu.pl}

\author{Robert Rałowski}
\email{robert.ralowski@pwr.edu.pl}

\author{Szymon Żeberski}
\email{szymon.zeberski@pwr.edu.pl}

\thanks{The work has been partially financed by grant {\bf 8211204601, MPK: 9130730000} from the Faculty of Pure and Applied Mathematics, Wrocław University of Science and Technology.
	}

\address{Marcin Michalski, Robert Rałowski, Szymon Żeberski, Faculty of Pure and Applied Mathematics, Wrocław University of Science and Technology, 50-370 Wrocław, Poland}

\date{}

\begin{document}

\begin{abstract}
	The motivation of this work are the two classical theorems on inscribing rectangles and squares into large subsets of the plane, namely Eggleston Theorem and Mycielski Theorem.
	
	Using Shoenfield Absoluteness Theorem we prove that for every Borel subset of the plane with uncountably many positive (with respect to measure or category) vertical section contains a rectangle $P\times B$ where $P$ is perfect and $B$ is Borel and positive. We also obtained a variant of Eggleston Theorem regarding the $\sigma$-ideal $\Ee$ generated by closed sets of measure zero.
	
	Furthermore we proved that every comeager (resp. conull) subset of the plane contains a rectangle $[T]\times H$, where $T$ is a Spinas tree containing a Silver tree and $H$ is comeager (resp. conull). Moreover we obtained a common generalization of Eggleston Theorem and Mycielski Theorem stating that every comeager (resp. conull) subset of the plane contains a rectangle $[T]\times H$ modulo diagonal, where $T$ is a uniformly perfect tree, $H$ is comeager (resp. conull) and $[T]\se H$.
\end{abstract}

\maketitle

\section{Introduction}

  The main motivation of this paper are the two following theorems on inscribing special kind rectangles and squares into large subsets of the plane.
   
\begin{theorem}[Eggleston \cite{Egg}] For every conull set $F\se [0,1]^2$ there are a perfect set $P\se [0,1]$ and conull $B\se [0,1]$ such that $P\times B \se F.$
\end{theorem}

\begin{theorem}[Mycielski \cite{Myc}]\label{Mycielski}
	For every comeager or conull set $X\se [0,1]^2$ there exists a perfect set $P\se [0,1]$ such that $P\times P\se X\cup \Delta$, where  $\Delta=\{(x,x): x\in [0,1]\}.$
\end{theorem}

  In \cite{Zeb} the author gave an alternative nonstandard proof of Eggleston Theorem. He also generalized it for subsets of the plane of positive measure. Analogous results were proved for the category.
  
  In \cite{BNT} the authors applied Eggleston Theorem to prove that the set of feebly continuity points of a Lebesgue measurable function $f: \RR^2\to \RR$ contains a rectangle of perfect sets.
  
  Another application of Eggelston Theorem appeared in  \cite{CieRal}. The authors showed that for every conull subset $B$ of a Polish (measure) space $X$ and an uncountable $\Gg_\delta$ subset $G$ of the space of measure preserving homeomorhpisms over $X$ there is a perfect set $P\se G$ such that the set $\bigcap_{f\in P}f[B]$ is conull. Similar results were also proved in the category case.

  Several directions of generalizing Mycielski Theorem were explored in \cite{MiRalZeb}, \cite{Spi1}, \cite{SolSpi} with various notions of largeness or more specific kind of perfect sets, e.g. superperfect sets. Theorem \ref{Mycielski} was also an inspiration for \cite{BanZdo} where authors showed that every comeager subset of the plane contains (modulo diagonal) a square of nowhere meager sets.

Our goal is to generalize Eggleston Theorem and conjoin it with Mycielski Theorem further developing methods used in \cite{Zeb} and \cite{MiRalZeb}. One of the directions of generalization is to meddle with the notion of largeness, i.e. replace the conull or comeager set with a set whose complement lies in another planar $\sigma$-ideal. The other direction is to replace the perfect set with a body of a certain kind of a perfect tree. This requires changing the underlying space to the one where trees grow. The most natural choice is the Cantor space $2^\omega$.

  We use standard set-theoretical notation following \cite{Jech}. Natural numbers are denoted by $\w$. For any set $A$ denote by
  \begin{itemize}
    \item $A^{<\omega}$ - the set of all finite sequences with members from $A$;
    \item $A^{\omega}$ - the set of all sequences with members from $A$;
    \item $[A]^{<\omega}$ - the set of finite subsets of $A$;
    \item $[A]^\omega$ - the set of infinite countable subsets of $A$;
    \item $P(A)$ - the power set of $A$.
  \end{itemize}
   
  Recall that $X$ is a Polish space if it is separable and completely metrizable. $\bor(X)$ denotes the family of Borel subsets of $X$. $\Ii$ and $\Jj$ denote $\sigma$-ideals of sets, i.e. families of sets closed under countable unions and taking subsets. We say that the family $\Aa$ is a base for $\Ii$ if for every set $A\in \Ii$ there is $B\in \Ii\cap \Aa$ such that $A\se B$. By $\bor(X)[\Ii]$ we denote $\sigma$-algebra of $\Ii$-measurable sets, i.e. sets of the form $B\symdif A$, where $B\in\bor(X)$ and $A\in \Ii$. We will consider well known $\sigma$-ideals, i.e. $\Mm(X)$ of meager subsets of $X$, $\Nn(X)$ of null subsets of $X$, $\Ee(X)$ generated by closed null subsets of $X$ and $\ctbl(X)$ of countable subsets of $X$. We will skip specifying the underlying space if the context is clear. All of these $\sigma$-ideals have Borel bases.
  
  There are certain cardinal coefficients associated with $\sigma$-ideals. In this paper we will use the following
  \begin{align*}
    \add(\Ii)&=\min\{|\Aa|: \Aa\se \Ii\,\land\,\bigcup\Aa\notin\Ii\},
    \\
    \cof(\Ii)&=\min\{|\Aa|: \Aa\se \Ii\,\land\,\Aa \tn{ is a base for }\Ii\}.
  \end{align*}
  We will mainly consider the Cantor space $2^\omega$. The topology of $2^\omega$ is generated by clopen sets of the form $[\sigma]=\{x\in 2^\omega:\; \sigma\se x\}$, where $\sigma\in 2^{<\w}$.
  
  We call a set $T\se 2^{<\w}$ a tree if for every $\sigma\in T$ and $n\in\dom(\sigma)$ it is the case that $\sigma\rest n\in T$. 
  \begin{definition}
		We say that a tree $T\se 2^{<\omega}$ is
		\begin{itemize}
		  \item perfect if $(\forall \sigma\in T) (\exists \tau\es \sigma)(\tau\concat 0, \tau\concat 1\in T)$;
		  \item uniformly perfect if it is perfect and 
		  \[
		    (\forall \sigma, \tau\in T)(|\sigma|=|\tau|\land \sigma\concat 0, \sigma\concat 1\in T \to \tau \concat 0, \tau\concat 1\in T);\]
		  \item a Silver tree if $T$ is perfect and 
		  \[
		    (\exists x\in 2^\omega)(\exists A\in [\omega]^\omega)(\forall \sigma\in T)(\forall n\in \dom(\sigma))(n\notin A \to \sigma(n)=x(n));
		  \]
			\item a Spinas tree if
				\[
					(\forall \tau\in T)(\exists N\in\omega) (\forall n\geq N) (\forall i\in 2) (\exists \tau'\in T\cap 2^{n+1})(\tau\se\tau'\,\land\,\tau'(n)=i). 
				\]
		\end{itemize}
	\end{definition}
  
  Notice that each Silver tree is uniformly perfect and each Spinas tree is perfect. A body of a tree $T\se 2^{<\w}$ is the set
  \[
    [T]=\{x\in 2^\omega:\; (\forall n\in \omega)(x\rest n\in T)\}.
  \]
  Bodies of perfect trees are perfect subsets of $2^\omega$. The notation for bodies of trees coincides with the one used for basic clopen sets, however we hope it will not lead to any confusion.
  
  Let $+$ be a coordinate wise addition modulo $2$. We will use this operation also for $A+x$, $\sigma+\tau$, $x+\sigma$, $A+\sigma$, where $A\se 2^\omega, x\in 2^\omega$ and $\sigma, \tau \in 2^{<\w}$. More precisely, let $\sigma\in 2^k, \tau=2^l, k\le l$. Then
  \begin{align*}
    A+x&=\{a+x:\; a\in A\};
    \\
    \sigma+\tau&=\tau+\sigma=\{(n, \sigma(n)+\tau(n)):\; n<k\}\cup\{(n, \tau(n):\; k\le n<l)\};
    \\
    x+\sigma&=((x\rest n)+\sigma)\cup (x\rest(\omega\bez n));
    \\
    A+\sigma&=\{x+\sigma:\; x\in A\}.
  \end{align*}
  
  For $n\in\omega$ let $\mbb{0}_n=(\underbrace{0, 0, \dots, 0}_{\tn{n - times}})$ and $\mbb{1}_n=(\underbrace{1, 1, \dots, 1}_{\tn{n - times}})$.

\section{Nonstandard proofs}
In this section we will focus on variations of Eggleston Theorem considering various notions of bigness. Proofs of the results will be based on Shoenfield Absoluteness Theorem.

By standard Polish spaces we mean countable products of $\baire, 2^\w, [0,1], \RR$ and $\Perf(\RR)$ - a space of perfect subsets of $\RR$. 

We say that $\varphi$ is $\Sigma_2^1$-sentence if for some canonical Polish spaces $X,Y$ and Borel set $B\subseteq X\times Y$ the sentence $\varphi$ is of the form:
$$
(\exists x\in X)(\forall y\in Y) (x,y) \in B.
$$
The Borel set $B$ has its so called Borel code $b\in\baire$ (see \cite{Kech}). The triple $(X, Y, b)$ is a parameter of $\Sigma^1_2$-sentence $\varphi$.
Now, let us recall Shoenfield Absoluteness Theorem.
\begin{theorem}[Shoenfield]\label{Shoenfield} Let $M\subseteq N$ be standard transitive models of ZFC and ${\w_1^N \subseteq M.}$ Let $\varphi$ be a $\Sigma_2^1$-sentence with a parameter from the model $M.$ Then
	\[
	M\models \varphi \iff N\models \varphi.
	\]
\end{theorem}
Let us recall that if $N$ is a generic extension of a standard transitive model $M$ of ZFC then $Ord^M = Ord^N$ and $\w_1^N\subseteq M.$

A method of providing nonstandard proofs of mentioned theorems will be as follows. We start with a standard transitive model $M$ of ZFC and find a generic extension $N$ of $M$ in which the theorem can be easily proved. Then we verify that the theorem forms a $\Sigma_2^1$-sentence. We apply Shoenfield Absoluteness Theorem to deduce that it is true in the ground universe $M.$


Let us recall that for ideals $\Ii\se P(X), \Jj\se P(Y)$ we define the Fubini product $\Ii\otimes\Jj$ of these ideals in the following way
\[
A\in \Ii\otimes\Jj \IFF (\exists B\in \tn{Bor}(X\times Y))(A\se B\;\land\; \{x\in X: B_x\notin\Jj\}\in\Ii),
\]
 $B_x=\{y\in Y: (x,y)\in B\}$ is a vertical section of the set $B$ (similarly we define a horizontal section $B^y$).

We say that $\Ii$ is  Borel--on--Borel if for every $B\in \bor(X\times X)$ the set 
$$\{x\in X:\; B_x\in\Ii\}$$
is Borel. Recall that $\Mm$ and $\Nn$ are Borel--on--Borel (see \cite{Kech}). 

As a tool we will use Cichoń Kamburelis and Pawlikowski theorem about cofinality of measure algebra, see \cite{CiKamPaw}.
\begin{theorem}[Cichoń-Kamburelis-Pawlikowski] There exists a dense subset of measure algebra $\bor(2^\w)/\Nn$ of cardinality of $\cof(\Nn).$ 
\end{theorem}

As a corollary we have the following fact.

\begin{fact}\label{fakt-CKP-miara} There exists a family $\Bb\se \perf(2^\w)\cap \Nn^+$ of size $\cof(\Nn)$ such that 
	$$
	(\forall A \in \bor(2^\w)[\Nn] \setminus\Nn) (\exists P\in \Bb) (P \se A).
	$$
\end{fact}
Notice that $\bor(2^\w)/\Mm$ contains a countable dense subset.
\begin{fact}\label{fakt-CKP-kategoria} There exists a family $\Bb\se G_\delta(2^\w)\cap \Mm^+$ of size $\cof(\Mm)$ such that 
	$$
	(\forall A \in \bor(2^\w)[\Mm] \setminus\Mm) (\exists P\in \Bb) (P \se A).
	$$
\end{fact}
Moreover, in Sacs model $\cof(\Nn) = \w_1<\w_2 = \c$ holds. 

In \cite{Zeb} the following generalization of Eggleston Theorem was proved via Shoenfield Absoluteness Theorem.
\begin{theorem}[Żeberski \cite{Zeb}] Let $P(2^\w)\es\Ii\in\{\Mm, \Nn\}$ and $G\se 2^\w\times 2^\omega$ be a Borel set such that $G\not\in \Ii\otimes \Ii.$ Then there are two sets $B,P\se2^\w$ such that $P\times B \subseteq G$, $P\in\perf(2^\w)$ and $B\in \bor(2^\w)\bez \Ii$.
\end{theorem}

  We will provide a generalization of this result as well as a new result concerning $\sigma$-ideal $\Ee$. The following series of notions and Lemmas will allow us to substantiate that formulas occurring in further results meet requirements of Shoenfield Absoluteness Theorem.
  
  Let $Y$ be a Polish space and $\Ii\se P(Y)$ be an ideal.
  \begin{definition}
    $\Ii$ has a good coding if there is a standard Polish space $X_\Ii$ and arithmetic formulas $\varphi_{\Ii}(x), \psi_\Ii(x,y)$ $(x\in X_\Ii, y\in Y)$ such that
    \[
      \{\{y: \psi_\Ii(x,y)\}:\;\varphi_\Ii(x)\}
    \]
    is a base of $\Ii$.
  \end{definition}
  The idea behind this definition is that $\varphi_\Ii(x)$ means that $x$ codes a basal set from ideal $\Ii$ and this set is exactly $\{y: \psi_\Ii(x,y)\}$.
\begin{lemma}\label{dobre kodowanie EMiN}
  $\Mm$, $\Nn$ and $\Ee$ have good codings.
\end{lemma}
\begin{proof}
  Let us start with $\Mm\se P(2^\omega)$. Let $X_\Mm=2^{\omega\times 2^{<\omega}}$. Then
  \begin{align*}
    \varphi_\Mm(x)=\;(\forall n\in\omega)(\forall \sigma\in 2^{<\w})(\exists \tau\in 2^{<\omega})(\sigma\se \tau \; \land\; x(n,\tau)=1).
  \end{align*}
  Moreover $\psi_\Mm(x,y) = \neg(\forall n)(\exists m) (x(n,y\rest m)=1)$.
  
  For the case of $\Nn\se P(2^\omega)$ let $X_\Nn = 2^{\omega\times 2^{<\omega}}$. Then
  \begin{align*}
    \varphi_\Nn(x)=\;(\forall n\in\omega)\large(\forall m\in\omega)\left(\sum\left\{\frac{1}{2^{|\sigma|}}:\; |\sigma|\le m \land x(n,\sigma)=1\right\}<\frac{1}{n}\right).
  \end{align*}
  Furthermore $\psi_\Nn(x,y) = (\forall n) (\exists m) (x(n,y\rest m)=1)$.
  
  Now let us consider the case of $\Ee\se P(2^\w)$. Set $X_\Ee = 2^{\omega\times 2^{<\omega}}$. Then
  \begin{align*}
    \varphi_\Ee(x)=\;&(\forall n\in\omega)\bigg((\forall \sigma, \tau \in 2^{<\w})(x(n,\sigma)=x(n,\tau)=1 \to \sigma \bot \tau)\land
    \\
    &\land(\forall k\in\omega)(\exists m\in\omega)\left(\sum\left\{\frac{1}{2^{|\sigma|}}:\; |\sigma|\le m \land x(n,\sigma)=1\right\}>1-\frac{1}{k}\right)\bigg).
  \end{align*}
  Moreover $\psi_\Ee(x,y) = \neg(\forall n) (\exists m) (x(n,y\rest m)=1)$.
  
\end{proof}

  For ideal $\Ii\se P(Y)$ let $\Ii^+=\Bor(Y)\bez \Ii$ be the family of Borel $\Ii$-positive sets. 
  \begin{definition}
    $\Ii^+$ has a good coding if there is a standard Polish space $X_{\Ii^+}$ and arithmetic formulas $\varphi_{\Ii^+}(x), \psi_{\Ii^+}(x,y)$ $(x\in X_{\Ii^+}, y\in Y)$ such that
    \[
      (\forall A\in \Ii^+) (\exists x \in X_{\Ii^+}) (\varphi_{\Ii^+}(x)\; \land\; \{y: \psi_{\Ii^+}(x,y)\} \se A).
    \]
  \end{definition}

We will use the following characterisation of positive Borel sets modulo ideal $\Ee$ from \cite[Lemma 2.11]{PPU}.
\begin{lemma}\label{tarasik}
	Let $A\se 2^\w$ be an analytic set such that $A\notin\Ee$. Then there exists a measure zero $G_\delta$-set $G$ such that $G\se A$ and the closure of $G$ has positive measure. 
\end{lemma}

\begin{lemma}\label{dobre kodowanie EMiN+}
  $\Mm^+$, $\Nn^+$ and $\Ee^+$ have good codings.
\end{lemma}
\begin{proof}
  Let us start with $\Mm^+\se P(2^\omega)$. Let $X_{\Mm^+}=2^{\omega\times 2^{<\omega}}$. Then
  \begin{align*}
    \varphi_{\Mm^+}(x)=\;(\exists \rho\in 2^{<\w})(\forall n\in\omega)(\forall \sigma\in 2^{<\w})(\rho \se \sigma \then (\exists \tau\in 2^{<\omega})(\sigma\se \tau \; \land\; x(n,\tau)=1)).
  \end{align*}
  Moreover $\psi_{\Mm^+}(x,y) = (\forall n)(\exists m) (x(n,y\rest m)=1)$.
  
  For the case of $\Nn^+\se P(2^\omega)$ let $X_\Nn = 2^{2^{<\omega}}$. Then
  \begin{align*}
    \varphi_{\Nn^+}(x)=\;(\exists k\in \omega)(\forall m\in\omega)\left(\sum\left\{\frac{1}{2^{|\sigma|}}:\; x(\sigma)=1\;\land\; |\sigma|\le m\right\}<1-\frac{1}{k}\right).
  \end{align*}
  Furthermore $\psi_{\Nn^+}(x,y)=\neg(\exists \sigma\in 2^{<\w})(x(\sigma)=1\;\land\;y\es \sigma)$.
  
  In the case of $\Ee^+\se P(2^\omega)$ we will use Lemma \ref{tarasik}. Set 
  $X_{\Ee^+}=X_{\Nn^+}\times X_{\Nn}=2^{2^{<\w}}\times 2^{\w\times 2^{<\w}}$.
  	Then for $x=(x_0,x_1)\in X_{\Ee^+}$
  	\begin{align*}
  		\varphi_{\Ee^+}(x_0,x_1&)=\; \varphi _{\Nn^+}(x_0)\land \varphi_{\Nn}(x_1) \land (\forall\rho\in 2^{<\w})(\neg (\exists \sigma_0,\sigma_1,\ldots, \sigma_n\in 2^{<\w})
  		\\
  		&([\rho]=[\sigma_0]\cup [\sigma_1]\cup\ldots\cup [\sigma_n]\land x_0(\sigma_0)=1\land  x_0(\sigma_1)=1\land\ldots\land x_0(\sigma_n)=1)\then
  		\\
  		&\then
  		(\forall n\in\w)(\exists\tau\in 2^{<\w})((x_1(n,\tau)=1\land \rho\se\tau ))\land
  		(\forall n\in\w)(\forall \tau\in 2^{<\w})
  		\\
  		&((x_1(n,\tau)=1\land \rho\se\tau )\then (\exists \tau'\in 2^{<\w})(\tau\se \tau'\land x_1(n+1,\tau')=1 )) .
  	\end{align*}
  	Moreover, $\psi_{\Ee^+}((x_0,x_1),y)=\psi_{\Nn}(x_1,y).$
\end{proof}

\begin{theorem}\label{shoenfield MN}
  Let $\Ii\in\{\Nn, \Mm\}$. Then for every set $G\in \bor(2^\w\times 2^\w) \setminus (\ctbl\otimes \Ii)$ there are $P\in \perf(2^\w)$ and $B\in \bor(2^\w)\setminus \Ii$ such that $P\times B \se G.$
\end{theorem}
\begin{proof} Let $V'$ be a generic extension of $\ZFC$ of a transitive model $V$ such that
	$$
	V'\models \aleph_1 = \cof(\Nn) < \c = \aleph_2.
	$$
	Let $G\in Bor(2^\w\times 2^\w)\setminus (\ctbl\otimes \Ii)$ coded in the ground universe $V.$ We work in $V'$. Define  
	$$
	  X = \{ x\in 2^\w:\; B_x\notin \Ii\}.
	$$   
	$X$ is uncountable. Furthermore, since $\Mm$ and $\Nn$ are Borel--on--Borel, ${\{ x\in 2^\w:\; B_x\in \Ii\}}$ is a Borel set. Therefore $X$ has cardinality $\c.$ 
	 
	 In the case of $\Ii=\Nn$ by Fact \ref{fakt-CKP-miara} there exists a family $\Bb\in \perf(2^\w)\setminus \Nn$ of size $\cof(\Nn) = \aleph_1$ cofinal in $\bor(2^\w)[\Nn]$. Since $X$ has size $\aleph_2=\c$,  there exists $Q\in \Bb$ such that
	$$
	|\{x\in 2^\w:\; Q\se G_x\}| = \aleph_2.
	$$
	Clearly, the above set is coanalytic and thus contains some perfect set $P\se 2^\w.$ Hence in the universe $V'$  there are perfect subsets $P,Q\se 2^\w$ with $\lambda(Q)>0$ such that $P\times Q \se G.$  $V'$ models the sentence 
	$$
	(\exists P\in \perf(2^\w ))(\exists x\in X_{\Nn^+}) (\forall y,z\in 2^\w)  (\varphi_{\Nn^+}(x) \land y\in P\land \psi_{\Nn^+}(x,z) \then (y,z) \in G),
	$$
	where $\varphi_{\Nn^+}, \psi_{\Nn^+}$ and $X_{\Nn^+}$ witness that ${\Nn^+}$ has a good coding by Lemma \ref{dobre kodowanie EMiN+}.
	It is $\Sigma^1_2$ sentence with a parameter from $V$. Hence, by Shoenfield Absoluteness Theorem, it is true in $V$.
	
	In the case $\Ii=\Mm$,  $X$ is uncountable. Since $\cof(\Mm)\le \cof(\Nn)$, it is true that  $V'\models\cof(\Mm)= \aleph_1$. Hence, the following sentence is true in $V'$
	$$
(\exists P\in \perf(2^\w ))(\exists x\in X_{\Mm^+}) (\forall y,z\in 2^\w)  (\varphi_{\Mm^+}(x) \land y\in P\land \psi_{\Mm^+}(x,z) \then (y,z) \in G),
$$
	where $\varphi_{\Mm^+}, \psi_{\Mm^+}$ and $X_{\Mm^+}$ witness that ${\Mm^+}$ has a good coding by Lemma \ref{dobre kodowanie EMiN+}.
	The proof is similar as in the first case (we use Fact \ref{fakt-CKP-kategoria} instead of Fact \ref{fakt-CKP-miara}) and use Shoenfield Absoluteness Theorem to come back to $V$.
	
	Notice that in both cases we obtain sentences implying the thesis of the theorem.
\end{proof}

We have the following immediate corollary regarding $\sigma$-ideal $\Mm\cap \Nn$.

\begin{corollary}
  For every set $G\in \bor(2^\w\times 2^\w) \setminus (\ctbl\otimes (\Nn\cap\Mm))$ there are $P\in \perf(2^\w))$ and $B\in \bor(2^\w)\setminus (\Nn\cap\Mm)$ such that $P\times B \se G.$
\end{corollary}


Now we will focus on $\sigma$-ideal $\Ee$ generated by closed null subsets of the Cantor space $2^\w$. In \cite{BarShe} Bartoszyński and Shelah  proved that 
$$
  \add(\Mm) = \add(\Ee) \;\; \& \;\; \cof(\Mm) =  \cof(\Ee).
$$
It is well known that under $\MA+\c=\aleph_3$ we have $\add(\Ee) = \add(\Mm) = \aleph_3.$ We will use the above results to prove the following theorem.

\begin{theorem}\label{shoenfield E} Let $G\in \Bor(2^\w\times 2^\omega)$ be such that $G^c\in \Ee\otimes \Ee$. Then there are 
	$P\in \perf(2^\w)$ and $B \in \bor(2^\w)$ satisfying $B^c\in \Ee$ and $P\times B\se G.$
\end{theorem}
\begin{proof} Let $V$ be a universe of ZFC such that $G\in V$ and let $V'$ be a forcing extension satisfying $\w_2 < \add(\Ee).$
	\\
	We work in $V'$. Let $Z=\{ x\in 2^\w:\; G_x^{c} \in \Ee\}$. Then $|Z| = \c\ge \w_3.$ Let us choose any set $Y\se Z$ of cardinality $\omega_2$. Since $\w_2 < \add(\Ee)$, the complement of a set $\bigcap_{y\in Y} G_y$ is in $\Ee.$ Let $B\in \Bor(2^\w)$ such that $B^c\in \Ee$, $B\subseteq \bigcap_{y\in Y} G_y$ and consider a set
	$
	A = \{ x\in2^\w:\; B\subseteq G_x \}.
	$
	Clearly, $A$ is coanalytic. Since $Y$ has cardinality $\w_2$ and $Y\subseteq A$,  $A$ contains a perfect subset $P.$ It implies that $V'$ is a model for the following formula
	\[
	(\exists x\in X_\Ee)(\exists P\in \Perf(2^\w))(\forall y,z\in2^\w)(y\in P \land \varphi_\Ee(x) \land \neg\psi_\Ee(x,z) \then (y,z)\in G).
	\]
	where $\varphi_\Ee, \psi_\Ee$ and $X_\Ee$ witness that $\Ee$ has a good coding by Lemma \ref{dobre kodowanie EMiN}. It is $\Sigma^1_2$, hence by Shoenfield Absoluteness Theorem it also holds in $V$.
\end{proof}
  
 Let us remark that another approach to obtained nonstandard proofs can be based on so called univesal set for bases of ideals. Such sets were invastigated by M. Michalski and A. Cieślak in \cite{CieMi}.
  
  In the light of Theorems \ref{shoenfield MN} and \ref{shoenfield E} a natural question arises regarding $\sigma$-ideal $\Ee$.
  \begin{question}
    Does every set $G\in \bor(2^\w\times 2^\w) \setminus (\ctbl\otimes \Ee)$ contain $P\times B$, where $P\in \perf(2^\w)$ and $B\in \bor(2^\w)\setminus \Ee$? 
  \end{question}

  Let us remark that the answer would be positive if $\Ee$ was Borel--on--Borel. 
  
  Now it would be a good moment to marry the concept of inscribing rectangles (Eggleston Theorem) and the concept of inscribing squares (Mycielski Theorem). A straightforward attempt via Shoenfield Absoluteness Theorem requires that every comeager (resp. conull) set can be separated from its meager (resp. null) subset by a Borel set. However, this is not the case as the following result shows.
  \begin{proposition}
    There exists a set $G\se 2^\omega, G^c\in \Ee$ and a set $X\se G, X\in \Ee$ such that there is no Borel set $B$ and no uncountable set $Y\se X$ such that $Y\se B \se G$.
  \end{proposition}
  \begin{proof}
    Let $P\se 2^\omega$ be a perfect null set and let $X$ be a relatively Bernstein set in $P$, i.e. $Q\cap X\ne\0$ and $X^c\cap Q\ne\0$ for every perfect $Q\se P$. Set $G=(2^\w\bez P)\cup X$. Let $Y\se X$ be uncountable and let $B$ be a Borel set containing $Y$. Then $B\cap P$ is uncountable, so $B\cap (P\bez X)\ne \0$. Hence, $B\not\se G$.
  \end{proof}
  Since $\Ee \se \Mm, \Nn$, the $\sigma$-ideal $\Ee$ can be replaced with either $\Mm$ or $\Nn$ and the Proposition will still hold true.

  We will tackle the subject of conjoining Eggleston Theorem and Mycielski Theorem with more conventional methods in the following sections.

\section{Category case}

In this section we will focus on generalizations of category variant of Eggleston Theorem, where the planar set and the vertical side remain comeager but the perfect set $P$ is replaced with a body of some type of perfect tree. Let us start with the case of Silver trees.

\begin{theorem}\label{Silver Eggleston} For every comeager set $G\se (2^\w\times 2^\w)$ there are a Silver tree $T\se 2^\w$ and a dense $G_\delta$-set $B\se2^\w$ such that $[T]\times B \se G.$
\end{theorem}
\begin{proof}
	Let us fix a topological base $\{S_n:\ n\in\w\}$ of the Cantor space $2^\w$.
	Wiothout loss of generality let $G\se 2^\w\times 2^\w$ be a dense $G_\delta$-set, $G=\bigcap_{n\in\w}U_n$ for some descending sequence $(U_n)_{n\in\w}$ of open dense sets.
	
	By induction, we will construct a sequence $(\tau_n)_{n\in\w}$ of elements of $2^{<\w}$ and a sequence $(V_n)_{n\in\w}$  of open subsets of $2^\w$ satysfying, for every natural number $n$, the following conditions:
	\begin{enumerate}
		\item $V_n\se S_n$;
		\item $[{\tau_0}\concat {i_0}\concat {\tau_1}\concat {i_1}\concat\ldots\concat{\tau_{n-1}}\concat {i_{n-1}}\concat{\tau_n}]\times V_n\se U_n$ for every {${(i_0,i_1,i_2,\ldots,i_{n-1})\in 2^n}$}.
	\end{enumerate} 
  
  To construct $\tau_0$ and $V_0$ it is enough to notice that the set $U_0\cap (2^\w\times S_0)$ is open and nonempty.
  
  Assume now that $(\tau_k)_{k<n}$ and $(V_k)_{k<n}$ are already constructed.
  \\
  For $i=(i_0,i_1,\ldots, i_{n-1})\in 2^n$ set $\tau(i)={\tau_0}\concat {i_0}\concat {\tau_1}\concat {i_1}\concat\ldots\concat{\tau_{n-1}}\concat {i_{n-1}}$.
  Consider the set
  $$
  W_n=\bigcap_{i\in 2^n}(([\tau(i)]\times S_n)\cap U_n)+(\tau(i)\concat 000\ldots,000\ldots).
  $$
  
  $W_n$ is open and dense in $[\mbb{0}_{|\tau(i)|}]\times S_n$. Hence, we can find $\tau_n$ and $V_n$ such that  ${[{\mbb{0}_{|\tau(i)|}}\concat {\tau_n} ]\times V_n\se W_n}$.
  The inductive construction is finished.
  
  Now, let us define
  \begin{align*}
    t&= {\tau_0}\concat {0}\concat {\tau_1}\concat {0}\concat{\tau_2}\concat {0}\concat{\tau_3}\concat\ldots,
    \\
    A&=\{\dom\tau_0, \dom\tau_0+\dom\tau_1+1,\dom\tau_0+\dom\tau_1+\dom\tau_2+2,\ldots \}.
  \end{align*}
  
  Notice that the set $\{x\in 2^\w:\ (\forall n\notin A) \,(x(n)=t(n))\}$ is a body of some Silver tree $T$.
  
  Set $B=\bigcap_{n\in\w}\bigcup_{m\ge n}V_m$. $B$ is a dense $G_\delta$ subset of $2^\w$.
  Moreover $[T]\times B\se G$. 
\end{proof}

Since every comeager subset of $2^\omega$ contains a body of a Silver tree, the above Theorem shows that every comeager subset of $2^\omega\times 2^\omega$ contains a rectangle of bodies of Silver trees.
On the other hand it cannot contain any square of bodies of Silver trees as it was shown in \cite[Proposition 2.5]{MiRalZeb} that there is an open dense set $U\se 2^\omega\times 2^\omega$ such that $[T]\times[T]\not\se U\cup\Delta$ for any Silver tree $T\se 2^{<\omega}$. 

  The following theorem generalizes Theorem \ref{Silver Eggleston}. Its proof, though more technical, follows the similar pattern outlined in the previous one.

\begin{theorem}\label{category spinas}
  For every comeager $G\se (2^\w\times 2^\w)$ there are a Spinas tree $T\se 2^{<\w}$ and a dense $G_\delta$-set $B\se2^\w$ such that $[T]\times B \se G.$ Moreover $T$ contains a Silver tree.
\end{theorem}
\begin{proof}
	As usual let us fix a topological base $\{S_n:\ n\in\w\}$ of the Cantor space $2^\w$.
	Without lose of generality $G$ is a dense $G_\delta$-set, $G=\bigcap_{n\in\w}U_n$ for some descending sequence $(U_n)_{n\in\w}$ of open dense sets.
	
	By induction, we will construct a sequence $(\tau_n)_{n\in\w}$ of elements of $2^{<\w}$ and a sequence $(V_n)_{n\in\w}$  of open subsets of $2^\w$ satisfying, for every natural number $n$, the following conditions:
	\begin{enumerate}
		\item $V_n\se S_n$;
		\item $[{\tau_0^{j_0\frown}}{i_0}\concat {\tau_1^{j_1\frown}} {i_1}\concat\ldots\concat{\tau_{n-1}^{j_{n-1}\frown}} {i_{n-1}}\concat{\tau_n^{j_n}}]\times V_n\se U_n$ for $i_k, j_k\in 2$, $k\in\{0,1,2,\ldots,n\}$,
	\end{enumerate} 
	where $\tau_k^{0}=\tau_k$ and $\tau_k^{1}=\mbb{1}_{|\tau_k|}-\tau_k$.
  
  Let $\widehat{U}_0=U_0\cap\{(x_n,y_n)_{n\in\omega}:\; (1-x_n,y_n)_{n\in\omega}\in U_0\}$. Notice that $\widehat{U}_0$ is open and dense, hence the set $\widehat{U}_0\cap (2^\w\times S_0)$ contains a rectangle $[\tau_0]\times V_0$.
  
  Now assume that $(\tau_k)_{k<n}$ and $(V_k)_{k<n}$ are already constructed.
  
  For $i=(i_0,i_1,\ldots, i_{n-1})\in 2^n$ and $j=(j_0,j_1,\ldots, j_{n-1})\in 2^n$ set
  \[
    \tau(i,j)={\tau_0^{j_0\frown}}{i_0}\concat {\tau_1^{j_1\frown}} {i_1}\concat\ldots\concat{\tau_{n-1}^{j_{n-1}\frown}} {i_{n-1}}.
    \]
  
  Let us denote $\widehat{U}_n=U_n \cap (({\mbb{0}_{|\tau(i,j)|}}\concat 11\ldots, 00\ldots)+U_n)$.
  Consider the set
  $$
  W_n=\bigcap_{i\in 2^n}\bigcap_{j\in 2^n}(([\tau(i,j)]\times S_n)\cap \widehat{U}_n)+(\tau(i,j)\concat 000\ldots,000\ldots).
  $$
  
  $W_n$ is open and dense in $[\mbb{0}_{|\tau(i,j)|}]\times S_n$. Hence, we can find $\tau_n$ and $V_n$ such that  
  \[
    {[{\mbb{0}_{|\tau(i,j)|}}\concat {\tau_n} ]\times V_n\se W_n}.
  \]
  
  The inductive construction is finished.
  
  Notice that the set 
  \[
    \{x\in 2^\w:\ (\exists i,j\in 2^\omega)(\forall n\in\omega)(\tau(i\rest n, j\rest n)\se x)\}
  \]
  
  is a body of some Spinas tree $T$ which contains a body of Silver tree.
  
  Set $B=\bigcap_{n\in\w}\bigcup_{m\ge n}V_m$. $B$ is a dense $G_\delta$ subset of $2^\w$.
  Moreover $[T]\times B\se G$.
\end{proof}

  The following Theorem is a successful mix of Eggleston Thoerem and Mycielski Theorem for uniformly perfect trees. Notice that it is not possible to incorporate Silver trees due to the remark made before Theorem \ref{category spinas}.

\begin{theorem}
    Let $G\se 2^\omega\times 2^\omega$ be comeager. Then there exist a uniformly perfect tree $T\se 2^{<\omega}$ and a dense $G_\delta$ set $D\se 2^\omega$ such that $[T]\se D$ and $[T]\times D\se G\cup \Delta$.
  \end{theorem}
  \begin{proof}
    Without loss of generality let us assume that $G=\bigcap_{n\in \omega}U_n$, where $(U_n: n\in\omega)$ is a descending sequence of open and dense subsets of $2^\omega\times 2^\omega$. Let $\{B_n: n\in \omega\}$ be a topological base of $2^\omega$. We will construct via induction sequences of open sets $(V_n: n\in \omega)$ and finite sequences $(\sigma_\tau: \tau \in 2^{<\omega})$ such that for $n\in\omega$ and $\tau, \tau'\in 2^{<\omega}$:
    \begin{enumerate}[(i)]
      \item ${\sigma_\tau}\concat i \se \sigma_{\tau\concat i}$ for $i\in \{0,1\}$;
      \item $|\tau|=|\tau'| \IMP |\sigma_\tau|=|\sigma_{\tau'}|$;
      \item $|\tau|=|\tau'|=n\; \land\; \tau\neq\tau' \IMP [\sigma_\tau] \times [\sigma_{\tau'}]\se U_n$;
      \item $V_n\se B_n$;
      \item $|\tau|=n \IMP [\sigma_\tau]\times V_n\se U_n$.
    \end{enumerate}
    For the step $0$ set $\sigma_\0\in 2^{<\omega}$ and open $V_0\se B_0$ in such a way that $[\sigma_\0]\times V_0\se U_0$. Next let us assume that we already have $(V_k: k\leq n)$ and $(\sigma_\tau: \tau\in 2^{\leq n})$ satisfying the conditions listed above. First let us find a suitable open set $V_{n+1}$. Let $\{\tau_k: k<2^n\}$ be an enumeration of $2^n$ in lexicographical order. For each $k<2^n$ and $i\in\{0,1\}$ we will pick $\sigma'_{{\tau_k}\concat i}$ and open sets $W_k^i\se B_{n+1}$ such that
    \begin{itemize}
      \item ${\sigma_{\tau_k}}\concat i\se \sigma'_{{\tau_k}\concat i}$;
      \item $[\sigma'_{{\tau_k}\concat i}]\times W_k^{i}\se U_{n+1}$;
      \item $W_j^0\es W_j^1$ and $W_j^i\es W_l^i$ for $j<l<2^n$.
    \end{itemize}
    We proceed by induction on $k<2^n$. Observe that $[{\sigma_{\tau_0}}\concat 0]\times B_{n+1}\cap U_{n+1}$ is a nonempty open set, thus it cantains a clopen rectangle. Let us denote it by $[\sigma'_{{\tau_0}\concat 0}]\times W_0^0$. Similarly, the set $[{\sigma_{\tau_0}}\concat 1]\times W_0^0\cap U_{n+1}$ is nonempty and open, hence contains a clopen rectangle which we denote by $[\sigma'_{{\tau_0}\concat 1}]\times W_0^1$. Assume that at the step $k+1<2^n$ we already have sequences $\sigma'_{{\tau_j}\concat i}$ and open sets $W_j^i$ for $j\leq k$ and $i\in\{0,1\}$ with desired properties. Then $[{\sigma_{\tau_{k+1}}}\concat 0]\times W_{k}^1\cap U_{n+1}$ contains a clopen rectangle, which we denote by $[\sigma'_{{\tau_{k+1}}\concat 0}]\times W_{k+1}^0$, and the set $[{\sigma_{\tau_{k+1}}}\concat 1]\times W_{k+1}^0\cap U_{n+1}$ contains a clopen rectangle, which we denote by $[\sigma'_{{\tau_{k+1}}\concat 1}]\times W_{k+1}^1$. We set $V_{n+1}=W_{2^n-1}^1$.
    \\
    We obtain $\sigma_\tau$ for $\tau\in 2^{n+1}$ by applying \cite[Lemma 11]{MiRalZeb} for $\{[\sigma'_\tau]: \tau\in 2^{n+1}\}$ and $U_{n+1}$. Finally we set 
    \begin{align*}
      P&=\bigcap_{n\in\omega}\bigcup_{\tau\in 2^n}[\sigma_\tau]
      \\
      D&=\bigcap_{n\in\omega}\bigcup_{k\geq n}(V_k\cup \bigcup_{\tau\in 2^k}[\sigma_{\tau}])
    \end{align*}
     Clearly, $P$ is a body of some uniformly perfect tree $T$, $D$ is dense $G_\delta$ and $P\se D$. To see that $P\times D\se G\cup \Delta$, let $(x,y)\in P\times D$, $x\neq y$ and $n\in\omega$. We will show that $(x,y)\in U_n$. Since $y\in D$ , there is $k\geq n$ such that $y\in V_k$ or $y\in[\sigma_\tau]$ for $\tau\in 2^k$. Furthermore, since $x\in P$, there is $\tau'\in 2^k$ such that $x\in[\sigma_{\tau'}]$. By (v) $[\sigma_{\tau'}]\times V_k\se U_k\se U_n$, so if $y\in V_k$ - we are done. If $y\in [\sigma_\tau]$ then by $x\neq y$ we may assume that $\sigma_\tau \perp \sigma_{\tau'}$, i.e. $\tau\neq \tau'$, hence by (iii) $[\sigma_{\tau'}]\times[\sigma_{\tau}]\se U_k$. This completes the proof.
  \end{proof}
  
  Let us finish this section with the following problem.
  
  \begin{question}
    Does every comeager set $G\se 2^\omega\times 2^\omega$ contain $([T]\times D)\bez \Delta$, where $T\se 2^{<\omega}$ is a Spinas tree and $D\se 2^\omega$ is a dense $G_\delta$ set such that $[T]\se D$?
  \end{question}

\section{Measure case}

The aim of the last section is to generalize Eggleston Theorem by replacing the prefect set with a body of some kind of perfect tree. In this context the notion of density $1$ points will be helpful.

For a set $A\se 2^\omega\times 2^\omega$ let
  \[
    (x,y)\in \widetilde{A} \IFF \lim_{n\to\infty}\frac{\lambda(A\cap [x\rest n]\times[y\rest n])}{2^{2n}}=1.
  \]
  $\widetilde{A}$ is the set of density $1$ points. If $A$ is closed then $\widetilde{A}\se A$ and $\lambda(\widetilde{A})=\lambda(A)$.
  
  The following Lemma may seem technical, however it will prove indispensable in all of the following proofs.
  
  \begin{lemma}\label{lemat menzurkowy}
    Let $\varepsilon>0$, $F\se 2^\omega$ closed, $\sigma\in 2^{<\omega}$, $H\se 2^\w$ a union of basic clopen sets of size $2^{-|\sigma|}$, satisfying $F\se [\sigma]\times H$ and $\lambda(F)>(1-\varepsilon^2)\lambda([\sigma]\times H)$. Then there exists $X\se [\sigma]$ satisfying $\lambda(X)>(1-\varepsilon)\lambda([\sigma])$ such that for each $x\in X$ 
    \begin{align*}\label{wzor menzurkowy}
      (\star)\quad &(\forall \delta>0)(\exists N\in\omega)(\forall n\ge N)(\exists S_{n}\se 2^n)
      \\
      &(\lambda(\bigcup_{\tau\in S_n}[\tau])>(1-\varepsilon)\lambda(H)\,\land\,(\forall \tau\in S_n)(\lambda(F\cap [x\rest n]\times [\tau])>(1-\delta)2^{-2n})).
    \end{align*}
  \end{lemma}
  \begin{proof}
    Let $X=\{x\in [\sigma]: \lambda(\widetilde{F}_x)>(1-\varepsilon)\lambda(H)\}$ and notice that $\lambda(X)>(1-\varepsilon)\lambda([\sigma])$ by Fubini Theorem. Fix $\delta>0$, $x\in X$ and define a function $f: (\widetilde{F})_x\to\omega$ in the following way
  \[
    f(y)=\min\left\{n\in\omega:\; (\forall m\geq n)( \lambda(F\cap ([x\rest m]\times [y\rest m]))> \frac{1}{2^{2m}}(1-\delta))\right\}.
  \]
  For a fixed $m\in\omega$ the function
  \[
    y\mapsto \lambda\big(F\cap([x \rest m]\times[y \rest m])\big)
  \]
  is continuous, hence $f$ is measurable. Notice that $(\widetilde{F})_x=\bigcup_{n\in\omega}f^{-1}[\{n\}]$, therefore there exists $N\in\omega$ such that 
  \[
    \lambda(Y_N)> \lambda(H)(1-\varepsilon), \tn{ where\quad} Y_N=\{y\in (\widetilde{F})_x:\; f(y)\leq N\}.
  \]
  For $n\ge N$ set 
  \[
    S_n=\{y\rest n: y\in (\widetilde{F})_x\;\land\;f(y)\leq N\}.
  \]
  Clearly, $Y_N\se \bigcup_{\tau\in S_n}[\tau]$ and $\lambda(\bigcup_{\tau\in S_n}[\tau])>(1-\varepsilon)\lambda(H)$. The condition
  \[
    (\forall \tau \in S_n )\,(\lambda(F\cap [x\rest n]\times[\tau])>(1-\delta)\frac{1}{2^{2n}})
  \]
  is a straightforward consequence of the definition of $f$.
  \end{proof}

Mirroring the category case, the following Theorem shows that every conull subset of $2^\omega\times 2^\omega$ contains a rectangle of bodies of Silver trees though it cannot contain any square of bodies of Silver trees, since  there exists a small set $A\se 2^\omega\times 2^\omega$ such that $(A\cap [T]\times [T])\setminus \Delta \neq\emptyset$ for any Silver tree $T\se 2^{<\omega}$ (see \cite[Proposition 3.4]{MiRalZeb}).

\begin{theorem}\label{measure case Silver}
  For every conull set $F\se (2^\w\times 2^\w)$ there are a Silver tree $T\se 2^{<\omega}$ and $F_\sigma$ conull set $H\se 2^\omega$ such that $[T]\times H \se F$.
\end{theorem}
\begin{proof}
  Let $(F_n)_{n\in\omega}$ be an ascending sequence of closed sets such that $\bigcup_{n\in\omega}F_n\se F$ and $\lim_{n\to\infty}\lambda(F_n)=1$.
  
  We will construct via induction on $k\in \omega$
  \begin{itemize}
    \item $N_k\in \omega$;
    \item $\sigma_\tau \in 2^{N_k},$ for every $\tau \in 2^k$;
    \item $H_{j,k}\se 2^\omega,$ which is a union of basic clopen sets of size $2^{-N_k}$ for $j=0,1,\dots , k$;
    \item $n_k\in \omega$;
  \end{itemize}
  satisfying for $\displaystyle\varepsilon_k=\frac{1}{2^{2k+2}(k+1)}$
  \begin{enumerate}[1)]
    \item\label{T jest perfect 1}  ${\sigma_{\tau}}\concat i\se \sigma_{\tau\concat i}$ for $\tau\in 2^k$ and $i=0,1$;
    \item\label{T jest Silver 1} $\sigma_{\tau}\rest(N_{k}, N_{k+1})=\sigma_{\tau'}\rest(N_k, N_{k+1})$ for $\tau, \tau'\in 2^{k+1}$;
    \item\label{warunek na H_j 1} $H_{k,k}=2^\w$ and $\lambda(H_{j,k+1})>\left(1-\frac{1}{2^{k+1}(k+1)}\right)\lambda(H_{j,k})$ for $j\le k$;
    \item\label{egg1} $(\forall \rho\in 2^{N_k})([\rho]\se H_{j,k}\to (\forall \tau\in 2^k)\left(\lambda(F_{n_j}\cap ([\sigma_\tau]\times[\rho]))>(1-\varepsilon_k^2)\lambda([\sigma_\tau]\times [\rho])\right)$
    \\
    and in consequence
    \[
      \lambda(F_{n_j}\cap ([\sigma_\tau]\times H_{j,k}))>(1-\varepsilon_k^2)\lambda([\sigma_\tau]\times H_{j,k}).
    \]
  \end{enumerate}
    
  Step $0$. Set $N_0=0$, $\sigma_\0=\0$, $H_{0,0}=2^\omega$ and let $n_0$ be such that $\lambda(F_{n_0})>(1-\varepsilon_0^2)$.
  
  Step $1$. We apply Lemma \ref{lemat menzurkowy} for $\varepsilon=\varepsilon_0$, $F=F_{n_0}, H=H_{0,0}, \sigma=\sigma_\0$ to obtain $X_{0,\0}$ such that $\lambda(X_{0,\0})>(1-\varepsilon_0)\lambda([\sigma_\0])=\frac{3}{4}$. Notice that $(X_{0,\0}+(0))\cap(X_{0,\0}+(1))\ne \0$, hence pick $x\in (X_{0,\0}+(0))\cap(X_{0,\0}+(1))$ and set $x_{(i)}=x+(i)$ for $i=0,1$. By Lemma \ref{lemat menzurkowy} we also get $N_{\0,0}$ and $N_{\0,1}$, associated with $x_{(0)}$ and $x_{(1)}$ respectively, that satisfy $(\star)$ for $\delta=\varepsilon_1^2$. Let $N_1=\max\{N_{\0,0}, N_{\0,1}\}$ and let $S^0_\0, S^1_\0\subseteq 2^{N_1}$  be sets of sequences corresponding to $x_{(0)}$ and $x_{(1)}$ respectively with $n=N_{1}$ according to Lemma \ref{lemat menzurkowy}.
\\
Set $H_{0,1,(i)}=\bigcup\{[\rho]:\; \rho\in S^i_\0\}$.
Notice that $H_{0,1}=H_{0,1,(0)}\cap H_{0,1,(1)}$ is a subset of $H_{0,0}$ and
  $$
    \lambda(H_{0,1})>(1-2\varepsilon_0)\lambda(H_{0,0})=\frac{1}{2}.
  $$
  This way the essential part of \ref{warunek na H_j 1}) for this step is satisfied. Moreover, for every $i\in\{0,1\}$
  $$
    \lambda(F_{n_0}\cap ([x_{(i)}\rest N_1]\times H_{0,1}))>(1-\varepsilon_1^2)\lambda([x_{(i)}\rest N_1]\times H_{0,1}).
  $$ 
  This takes care of the main part of \ref{egg}). For $i\in\{0,1\}$ let us define $\sigma_{(i)}=x_{(i)}\rest N_1$. Clearly, this satisfies conditions \ref{T jest perfect 1}) and \ref{T jest Silver 1}).
  \\
  Pick $n_1>n_0$ such that for every $i\in\{0,1\}$ 
  \begin{align*}
    \lambda(F_{n_1}\cap ([\sigma_{(i)}]\times 2^\w))&>(1-\varepsilon_{1}^2)\lambda([\sigma_{(i)}]\times 2^\w).
  \end{align*}
  Set $H_{1,1}=2^\omega$. Now both conditions \ref{warunek na H_j}) and \ref{egg}) are satisfied.

  Step $k+1$. For each $j\le k$ we apply Lemma \ref{lemat menzurkowy} for $\varepsilon=\varepsilon_k$, $F=F_{n_j}, H=H_{j,k}, \sigma=\sigma_\tau$, $\tau\in 2^k$, to obtain $X_{j,\tau}\subseteq [\sigma_\tau]$ such that $\lambda(X_{j,\tau})>(1-\varepsilon_k)\lambda([\sigma_\tau])$. 
  \\
  Set $X_{k+1,\tau}=\bigcap_{j=0}^k X_{j,\tau}$ and $X_{k+1}=\bigcap_{\tau\in 2^k}\bigcap_{i\in\{0,1\}}(X_{k+1, \tau}+{\sigma_\tau}\concat i)$. Notice that
  \[
    \lambda(X_{k+1})>(1-2^{k+1}(k+1)\varepsilon_k)2^{-N_k}>0.
  \]
  Pick $x\in X_{k+1}$ and for $i=0,1$ denote $x_{\tau\concat i}=x+{\sigma_{\tau}}\concat i\in X_{k+1,\tau}$ along with associated $N_{\tau,i}$ satisfying $(\star)$ for $\delta=\varepsilon_{k+1}^2$. Set
  \[
    N_{k+1}=\max\{N_{\tau, i}:\; i\in\{0,1\},\ \tau\in 2^k\}.
  \]
  Let $S^0_\tau, S^1_\tau\subseteq 2^{N_{k+1}}$  be sets of sequences corresponding to $x_{\tau\concat 0}$ and $x_{\tau\concat 1}$ respectively with $n=N_{k+1}$ according to Lemma \ref{lemat menzurkowy}.
Set $H_{j,k+1,\tau\concat i}=\bigcup\{[\rho]:\; \rho\in S^i_\tau\}$ for $i\in\{0,1\}$.
Notice that $H_{j,k+1}=\bigcap\{H_{j,k+1,\tau}:\ \tau\in 2^{k+1}\}$ is a subset of $H_{j,k}$ and
  $$
    \lambda(H_{j,k+1})>(1-2^{k+1}\varepsilon_k)\lambda(H_{j,k}).
  $$
  This satisfies the essential part of \ref{warunek na H_j 1}) for this step. Moreover, for every $i\in\{0,1\}$ and $\tau \in 2^k$
  $$
    \lambda(F_{n_j}\cap ([x_{\tau\concat i}\rest N_{k+1}]\times H_{j,k+1}))>(1-\varepsilon_{k+1}^2)\lambda([x_{\tau\concat i}\rest N_{k+1}]\times H_{j,k+1}).
  $$
  This takes care of the main part of condition \ref{egg}). For $i\in\{0,1\}$ and $\tau\in 2^k$ let us define $\sigma_{\tau\concat i}=x_{\tau\concat i}\rest N_{k+1}=(x+{\sigma_\tau}\concat i)\rest N_{k+1}$. Notice that conditions \ref{T jest perfect 1}) and \ref{T jest Silver 1}) are met.
  \\
   Pick $n_{k+1}>n_k$ such that for every $\tau\in 2^{k+1}$ 
  \begin{align*}
    \lambda(F_{n_{k+1}}\cap ([\sigma_\tau]\times 2^\w))&>(1-\varepsilon_{k+1}^2)\lambda([\sigma_\tau]\times 2^\w).
  \end{align*}
  Set $H_{k+1,k+1}=2^\omega$. Now conditions \ref{warunek na H_j 1}) and \ref{egg1}) for this step are fully satisfied.
  
  The construction is complete.
  \\
  Set $T=\{\rho\in 2^{<\w}:\; (\exists \tau\in 2^{<\w})(\rho\se\sigma_\tau)\}$, $H=\bigcup_{j}H_j$, where $H_j=\bigcap_{k\ge j}H_{j,k}$.
  \\
  $H$ is conull. Let notice that for $j\in \omega$ by \ref{warunek na H_j 1})
  \[
    \lambda(H_j)\ge \lambda(H_{j,j})\prod_{k\ge j}(1-\frac{1}{2^{k+1}(k+1)})=\prod_{k\ge j}(1-\frac{1}{2^{k+1}(k+1)}).
  \]
  Furthermore the product $\prod_{k\ge 0}(1-\frac{1}{2^{k+1}(k+1)})$ is convergent, hence
  \[
    \lim_{j\to\infty}\prod_{k\ge j}(1-\frac{1}{2^{k+1}(k+1)})=1.
  \]
  By \ref{T jest perfect 1}) and \ref{T jest Silver 1}) $T$ is a Silver tree.
  \\
  To show that $[T]\times H\se F$ we will prove that $[T]\times H_j\se F_{n_j}$ for each $j\in\omega$. Pick any $(t,h)\in [T]\times H_j$. Notice that $[T]=\bigcap_{k\in\omega}\bigcup_{\tau\in 2^k}[\sigma_\tau]$. By \ref{egg}) for each $k\ge j$ there are (unique) $\alpha_k, \beta_k\in 2^{N_k}$ such that $\alpha_k\in\{\sigma_\tau: \tau \in 2^k\}$, $[\beta_k]\se H_{j,k}$ and
  \[
    (t,h)\in [\alpha_k]\times[\beta_k] \tn{\; and \;} ([\alpha_k]\times[\beta_k])\cap F_{n_j}\ne \0.
  \]
  Since $\bigcap_{k\ge j}[\alpha_k]\times[\beta_k]=\{(t,h)\}$ and $F_{n_j}$ is closed, $(t,h)\in F_{n_j}$.
\end{proof}

Even more general statement is true involving Spinas trees.

\begin{theorem}
  For every set $F\se (2^\w\times 2^\w)$ there are a Spinas tree $T\se 2^{<\w}$ and $F_\sigma$ conull set $B\se 2^\omega$ such that $[T]\times B \se F$. Moreover, $T$ contains a Silver tree.
\end{theorem}

The proof is a natural modification of the proof of Theorem \ref{measure case Silver} borrowing some bits from the proof of Theorem \ref{category spinas}. The authors believe proving it in detail is neither helpful nor necessary.

Again, analogously to the category case, relaxing the requirement on the type of tree yields us the Eggleston-Mycielski result for uniformly perfect trees.
  
\begin{theorem} 
  For every conull set $F\se (2^\w\times 2^\w)$ there are a uniformly perfect tree $T\se 2^{<\omega}$ and $F_\sigma$ conull set $B\se 2^\omega$ such that $[T]\se B$ and $[T]\times B \se F\bez \Delta$.
\end{theorem}
\begin{proof}
  Let $(F_n)_{n\in\omega}$ be an ascending sequence of closed sets such that $\bigcup_{n\in\omega}F_n\se F$ and $\lim_{n\to\infty}\lambda(F_n)=1$. Let $d(\tau, \tau')=\min\{n\in\omega : \; \tau(n)\ne \tau(n')\}$ and $d(\tau,\tau)=|\tau|$.
  
  We will construct via induction on $k\in \omega$
  \begin{itemize}
    \item $N_k\in \omega$;
    \item $\sigma_\tau \in 2^{N_k},$ for every $\tau \in 2^k$;
    \item $H_{j,k}\se 2^\omega,$ which is a union of basic clopen sets of size $2^{-N_k}$ for $j=0,1,\dots , k$;
    \item $n_k\in \omega$;
  \end{itemize}
  satisfying for $\displaystyle\varepsilon_k=\frac{1}{2^{2k+2}(k+1)}$
  \begin{enumerate}[1)]
    \item\label{T jest perfect}  $\sigma_{\tau\concat 0}, \sigma_{\tau\concat 1} \es \sigma_{\tau}$ and $\sigma_{\tau\concat 0}\ne \sigma_{\tau\concat 1}$ for $\tau\in 2^{k}$;
    \item\label{T jest uniformly} $d(\sigma_{\tau\concat 0}, \sigma_{\tau\concat 1})
    =d(\sigma_{\tau'{\concat 0}}, \sigma_{\tau' {\concat 1}})$ for $\tau, \tau'\in 2^k$;
    \item\label{warunek na H_j} $H_{k,k}=2^\w$ and $\lambda(H_{j,k+1})>\left(1-\frac{1}{2^{k+1}(k+1)}\right)\lambda(H_{j,k})$ for $j\le k$;
    \item\label{egg} $(\forall \rho\in 2^{N_k})([\rho]\se H_{j,k}\to (\forall \tau\in 2^k)\left(\lambda(F_{n_j}\cap ([\sigma_\tau]\times[\rho]))>(1-\varepsilon_k^2)\lambda([\sigma_\tau]\times [\rho])\right)$
    \\
    and in consequence
    \[
      \lambda(F_{n_j}\cap ([\sigma_\tau]\times H_{j,k}))>(1-\varepsilon_k^2)\lambda([\sigma_\tau]\times H_{j,k}).
    \]
    \item\label{myc} $\lambda(([\sigma_\tau]\times[\sigma_\tau'])\cap F_{n_{d(\tau,\tau')}})>(1-\varepsilon_k)2^{-2N_k}$ for $j\le k$;
  \end{enumerate}
    
  Step $0$. Set $N_0=0$, $\sigma_\0=\0$, $H_{0,0}=2^\omega$ and let $n_0$ be such that $\lambda(F_{n_0})>(1-\varepsilon_0^2)$.
  
  Step $1$. We apply Lemma \ref{lemat menzurkowy} for $\varepsilon=\varepsilon_0$, $F=F_{n_0}, H=H_{0,0}, \sigma=\sigma_\0$ to obtain $X_{0,\0}$ such that $\lambda(X_{0,\0})>(1-\varepsilon_0)\lambda([\sigma_\0])=\frac{3}{4}$. For any set $A\se 2^\w\times 2^\w$ denote $A^s=A\cap A^{-1}$. Notice that
  \[
    \lambda(F_{n_0}^s)>1-2\varepsilon_0^2=\frac{14}{16}.
  \]  
  Hence
  \[
    \lambda(F_{n_0}^s\cap X_{0,\0}^2)>\frac{7}{16}
  \]
  and there are distinct $x_{(0)}$, $x_{(1)}\in X_{0,\0}$ with $(x_{(0)}$, $x_{(1)})\in \widetilde{F_{n_0}^s}$. By Lemma \ref{lemat menzurkowy} we also get $N_{\0,0}$ and $N_{\0,1}$, associated with $x_{(0)}$ and $x_{(1)}$ respectively, that satisfy $(\star)$ for $\delta=\varepsilon_1^2$. Choose $N_1\ge N_{\0,0}, N_{\0,1}$ such that
  \[
    (\forall n\ge N_{1})(\lambda(([x_{(0)}\rest n]\times[x_{(1)}\rest n])\cap F_{n_0}^s)>(1-\varepsilon_1)\frac{1}{2^n}).
  \]
  
  This inequality ensures that for this step the condition \ref{myc}) will be met. Let $S^0_\0, S^1_\0\subseteq 2^{N_1}$  be sets of sequences corresponding to $x_{(0)}$ and $x_{(1)}$ respectively with $n=N_{1}$ according to Lemma \ref{lemat menzurkowy}.

Set $H_{0,1,(i)}=\bigcup\{[\rho]:\; \rho\in S^i_\0\}$.
Notice that $H_{0,1}=H_{0,1,(0)}\cap H_{0,1,(1)}$ is a subset of $H_{0,0}$ and
  $$
    \lambda(H_{0,1})>(1-2\varepsilon_0)\lambda(H_{0,0})=\frac{1}{2}.
  $$
  This way the essential part of \ref{warunek na H_j}) for this step is satisfied. Moreover, for every $i\in\{0,1\}$
  $$
    \lambda(F_{n_0}\cap ([x_{(i)}\rest N_1]\times H_{0,1}))>(1-\varepsilon_1^2)\lambda([x_{(i)}\rest N_1]\times H_{0,1}).
  $$ 
  This takes care of the main part of \ref{egg}). For $i\in\{0,1\}$ let us define $\sigma_{(i)}=x_{(i)}\rest N_1$. Clearly, this satisfies conditions \ref{T jest perfect}) and \ref{T jest uniformly}).
  
  Pick $n_1>n_0$ such that for every $i\in\{0,1\}$ 
  \begin{align*}
    \lambda(F_{n_1}\cap ([\sigma_{(i)}]\times 2^\w))&>(1-\varepsilon_{1}^2)\lambda([\sigma_{(i)}]\times 2^\w),
    \\
    \lambda([\sigma_{(i)}]^2\cap F_{n_{1}})&>(1-\varepsilon_{1})2^{-2|\sigma_{(i)}|}.
  \end{align*}
  Set $H_{1,1}=2^\omega$. Now all conditions \ref{warunek na H_j}), \ref{egg}) and \ref{myc}) are satisfied.

  Step $k+1$. For each $j\le k$ we apply Lemma \ref{lemat menzurkowy} for $\varepsilon=\varepsilon_k$, $F=F_{n_j}, H=H_{j,k}, \sigma=\sigma_\tau$, $\tau\in 2^k$, to obtain $X_{j,\tau}\subseteq [\sigma_\tau]$ such that $\lambda(X_{j,\tau})>(1-\varepsilon_k)\lambda([\sigma_\tau])$.
  \\
Notice that $X_{k+1,\tau}=\bigcap_{j=0}^k X_{j,\tau}$ is a subset of $[\sigma_\tau]$ and
  $$
    \lambda(X_{k+1,\tau})>(1-(k+1)\varepsilon_k)\lambda([\sigma_\tau]).
  $$  
  From the previous step for for $\tau, \tau'\in 2^k$ and $d(\tau, \tau')=l$ we have 
  \[
    \lambda(([\sigma_\tau]\times [\sigma_{\tau'}])]\cap F_{n_l})>(1-\varepsilon_k)\lambda({[\sigma_\tau]})^2.
  \]
  Set $X_{k+1}=\bigcap_{\tau\in 2^k}(X_{k+1, \tau}+\sigma_\tau)$ and notice that
  \[
    \lambda(X_{k+1})>(1-2^k(k+1)\varepsilon_k)2^{-N_k}.
  \]
  Define
  \[
    R_{k+1}=\bigcap_{\tau,\tau'\in 2^k}(([\sigma_\tau]\times[\sigma_{\tau'}])\cap F_{n_{d(\tau, \tau')}})+(\sigma_{\tau},\sigma_{\tau'}))^s.
  \]
  and see that
  \[
    \lambda(R_{k+1})>(1-2^{2k+1}\varepsilon_k)2^{-2N_{k}}.
  \]
  Since $X_{k+1}^2, R_{k+1}\se [\mathbb{0}_{N_k}]^2$ and 
  \[
    \lambda(X_{k+1}^2)+\lambda(R_{k+1})>\lambda([\mathbb{0}_{N_k}])^2
  \]
  the set $X_{k+1}^2\cap R_{k+1}$ has positive measure. Pick $(x_0,x_1)\in \widetilde{R_{k+1}}\cap X_{k+1}^2$, $x_0\ne x_1$. For $i=0,1$ denote $x_{\tau\concat i}=x_i+\sigma_{\tau}\in X_{k+1,\tau}$ along with associated $N_{\tau,i}$ satisfying $(\star)$ for $\delta=\varepsilon_{k+1}^2$. Set
  \[
    N_{k+1}\ge\max\{N_{\tau, i}:\; i\in\{0,1\},\ \tau\in 2^k\}
  \]
  such that
  \[
    (\forall n\ge N_{k+1})(\lambda(([x_0\rest n]\times[x_1\rest n])\cap R_{k+1})>(1-\varepsilon_{k+1})\frac{1}{2^{2n}}). 
  \]
  This way the crucial part of \ref{myc}) for this step is satisfied. Let $S^0_\tau, S^1_\tau\subseteq 2^{N_{k+1}}$  be sets of sequences corresponding to $x_{\tau\concat 0}$ and $x_{\tau\concat 1}$ respectively with $n=N_{k+1}$ according to Lemma \ref{lemat menzurkowy}.
Set $H_{j,k+1,\tau\concat i}=\bigcup\{[\rho]:\; \rho\in S^i_\tau\}$ for $i\in\{0,1\}$.
Notice that $H_{j,k+1}=\bigcap\{H_{j,k+1,\tau}:\ \tau\in 2^{k+1}\}$ is a subset of $H_{j,k}$ and
  $$
    \lambda(H_{j,k+1})>(1-2^{k+1}\varepsilon_k)\lambda(H_{j,k}).
  $$
  
  This satisfies the essential part of \ref{warunek na H_j}) for this step. Moreover, for every $i\in\{0,1\}$ and $\tau \in 2^k$
  $$
    \lambda(F_{n_j}\cap ([x_{\tau\concat i}\rest N_{k+1}]\times H_{j,k+1}))>(1-\varepsilon_{k+1}^2)\lambda([x_{\tau\concat i}\rest N_{k+1}]\times H_{j,k+1}).
  $$
  
  This takes care of the main part of condition \ref{egg}). For $i\in\{0,1\}$ and $\tau\in 2^k$ let us define $\sigma_{\tau\concat i}=x_{\tau\concat i}\rest N_{k+1}=(x_i+\sigma_\tau)\rest N_{k+1}$. Notice that conditions \ref{T jest perfect}) and \ref{T jest uniformly}) are met.

   Pick $n_{k+1}>n_k$ such that for every $\tau\in 2^{k+1}$ 
  \begin{align*}
    \lambda(F_{n_{k+1}}\cap ([\sigma_\tau]\times 2^\w))&>(1-\varepsilon_{k+1}^2)\lambda([\sigma_\tau]\times 2^\w),
    \\
    \lambda([\sigma_\tau]^2\cap F_{n_{k+1}})&>(1-\varepsilon_{k+1})\frac{1}{2^{2|\sigma_\tau|}}.
  \end{align*}
  
  Set $H_{k+1,k+1}=2^\omega$. Now conditions \ref{warunek na H_j}), \ref{egg}) and \ref{myc}) for this step are fully satisfied.
  
  The construction is complete.
  
  Set $T=\{\rho\in 2^{<\w}:\; (\exists \tau\in 2^{<\w})(\rho\se\sigma_\tau)\}$, $H=\bigcup_{j}H_j$, where $H_j=\bigcap_{k\ge j}H_{j,k}$ and $B=[T]\cup H$.
  
  The prove that $H$ is conull is exactly the same as in the previous Theorem.
  
  By \ref{T jest perfect}) and \ref{T jest uniformly}) $T$ is uniformly perfect.
  
  To demonstrate that $[T]\times B\se F\bez\Delta$ it suffices to show that $[T]\times H\se F$ and $[T]\times[T]\se F\bez\Delta$.
  
  The proof of the former is almost identical to the proof of the analogous fact in the previous Theorem.
  
  To see the latter, choose any $(t,t')\in [T]^2$, $t\ne t'$. Let 
  \[
    j=\min\{n\in\omega:\; (\exists \tau, \tau'\in 2^n, \tau\ne \tau')(\sigma_{\tau}\se t\,\land\, \sigma_{\tau'}\se t')\}
   \]
For every $k\ge j$ there are $\tau_k, \tau'_k\in 2^k$, $\tau_k\ne\tau'_k$, such that $(t,t')\in [\sigma_{\tau_k}]\times[\sigma_{\tau'_k}]$. By \ref{myc}) $[\sigma_{\tau_k}]\times[\sigma_{\tau'_k}]\cap F_{n_j}\ne \0$. Since $\bigcap_{k\ge j}[\sigma_{\tau_k}]\times[\sigma_{\tau'_k}]=\{(t,t')\}$ and $F_{n_j}$ is closed, $(t,t')\in F_{n_j}$.
\end{proof}

Let us finish the paper with the following problem.

\begin{question}
    Does every conull set $G\se 2^\omega\times 2^\omega$ contain $([T]\times H)\bez \Delta$, where $T\se 2^{<\omega}$ is a Spinas tree and $H\se 2^\omega$ is a conull $F_\sigma$ set such that $[T]\se H$?
  \end{question}

\printbibliography

\end{document}